\documentclass[12pt]{article}%
\usepackage{amsmath}
\usepackage{amsfonts}
\usepackage{amssymb}
\usepackage{cite}
\usepackage{graphicx}%
\providecommand{\U}[1]{\protect\rule{.1in}{.1in}}
\newtheorem{theorem}{Theorem}[section]

\newtheorem{algorithm}{Algorithm}[section]

\newtheorem{case}{Case}[section]

\newtheorem{condition}{Condition}[section]

\newtheorem{corollary}{Corollary}[section]

\newtheorem{definition}{Definition}[section]

\newtheorem{lemma}{Lemma}[section]

\newtheorem{proposition}{Proposition}[section]
\newtheorem{remark}{Remark}[section]

\newenvironment{proof}[1][Proof]{\begin{trivlist}
\item[\hskip \labelsep {\bfseries
#1}]}{\hspace*{\fill}$\square$\end{trivlist}}
\makeatletter
\renewcommand\@biblabel[1]{#1.}
\makeatother
\begin{document}

\title{\textbf{Strict Fej\'{e}r Monotonicity by Superiorization of
Feasibility-Seeking Projection Methods}}
\author{Yair Censor\thanks{Y. Censor (corresponding author) Department of Mathematics,
University of Haifa, Mt. Carmel, Haifa 3498838, Israel. e-mail:
yair@math.haifa.ac.il} $\ \ \bullet$ Alexander J. Zaslavski\thanks{A.J.
Zaslavski, Department of Mathematics, The Technion -- Israel Institute of
Technology, Technion City, Haifa 32000, Israel. e-mail:
ajzasl@techunix.technion.ac.il}
\and Communicated by Jonathan Michael Borwein}
\date{February 23, 2014. Revised April 19, 2014. Revised: May 26, 2014.}
\maketitle

\begin{abstract}
We consider the superiorization methodology, which can be thought of as lying
between feasibility-seeking and constrained minimization. It is not quite
trying to solve the full fledged constrained minimization problem; rather, the
task is to find a feasible point which is superior (with respect to the
objective function value) to one returned by a feasibility-seeking only algorithm.

Our main result reveals new information about the mathematical behavior of the
superiorization methodology. We deal with a constrained minimization problem
with a feasible region, which is the intersection of finitely many closed
convex constraint sets, and use the dynamic string-averaging projection
method, with variable strings and variable weights, as a feasibility-seeking
algorithm. We show that any sequence, generated by the superiorized version of
a dynamic string-averaging projection algorithm, not only converges to a
feasible point but, additionally, either its limit point solves the
constrained minimization problem or the sequence is strictly Fej\'{e}r
monotone with respect to a subset of the solution set of the original problem.

\end{abstract}

\textbf{Keywords }Bounded perturbation resilience, constrained minimization,
convex feasibility problem, dynamic string-averaging projections, strict
Fej\'{e}r monotonicity, subgradients, superiorization methodology,
superiorized version of an algorithm.

\textbf{2010 Mathematics Subject Classification (MSC) }90C25, 90C30, 90C45, 65K10.

\section{Introduction\label{sect:intro}}

\textbf{What is superiorization}. The recently developed superiorization
methodology (SM) lies between feasibility-seeking and constrained minimization
(CM). It is not quite trying to solve the full fledged CM problem; rather, the
task is to find a feasible point of the CM problem, that is superior, not
necessarily optimal, with respect to the objective function value to one
returned by a feasibility-seeking only algorithm. Therefore, the SM can be
beneficial for CM problems for which an exact algorithm has not yet been
discovered, or when existing exact optimization algorithms are very time
consuming or require too much computer space for realistic large problems to
be run on commonplace computers. In such cases, efficient feasibility-seeking
iterative projection methods that provide non-optimal but
constraints-compatible solutions, can be turned by the SM into efficient
algorithms for superiorization that will be practically useful from the point
of view of the underlying objective function.

For the SM to be useful for a CM problem we need to have an efficient
feasibility-seeking algorithm that is in some well-defined sense perturbation
resilient. Then the SM uses those permitted perturbations in order to steer
the superioized version of the original feasibility-seeking algorithm toward
points with lesser, not necessarily minimal, objective function values. The
advantage is that in this manner one uses essentially not an optimization
algorithm but a superiorized-feasibility-seeking algorithm to attack the CM
problem. The latter methods are in many cases very efficient; see, e.g.,
\cite{cccdh10}, and, therefore, can save time and computing resources as
compared with exact optimization algorithms.

Additionally, in many mathematical formulations of significant real-world
technological or physical problems the objective function is exogenous to the
modeling process which defines the constraints. In such cases the
\textquotedblleft faith\textquotedblright\ of the modeler in the usefulness of
an objective function for the application at hand is limited and, as a
consequence, it is not worthwhile to invest too much resources in trying to
reach an exact constrained minimum point. These notions are rigorously
explained in the next sections below.

\textbf{Contribution}. Our main result, in Theorem \ref{thm:2.1} below,
establishes a mathematical basis for the behavior of the SM when dealing with
a CM problem with a feasible region that is the intersection of finitely many
closed convex constraint sets, see Case \ref{case:1} in Section \ref{sect:SM}
below. We use the dynamic string-averaging projection (DSAP) method, with
variable strings and variable weights, algorithmic scheme as a
feasibility-seeking algorithm, which is indeed bounded perturbations
resilient. The bounded perturbations resilience of the DSAP method has been
proved in \cite{cz12} and the practical behavior of the SM was observed in
numerous recent works, see references mentioned below. Our contribution here
is the mathematical guarantee of the convergence behavior of the superiorized
version of the DSAP algorithm.

Theorem \ref{thm:2.1} below says that any sequence, generated by the
superiorized version of a DSAP algorithm, given in Algorithm
\ref{alg:super-dsap} below, will not only converge to a feasible point of the
underlying CM problem, a fact which is due to the bounded perturbations
resilience of the DSAP method, but, additionally, that either its limit point
will solve the CM problem (\ref{eq:cm}) or that the sequence is strictly
Fej\'{e}r monotone with respect to; i.e., gets strictly closer to the points
of, a subset of the solution set of the CM problem according to
(\ref{eq:strict-F}) below.

\textbf{Related work}. This paper is a sequel to a series of recent
publications on the SM
\cite{cz12,bk13,bdhk07,cdh10,dhc09,gh13,rand-conmath,hd08,hgdc12,wenma13,ndh12,pscr10}%
, culminating in \cite{cdhst14}. The latter contains a detailed description of
the SM, its motivation, and an up-to-date review of SM-related previous work,
including a reference to \cite{bdhk07} in which it all started, although
without using yet the terms superiorization and bounded perturbation
resilience. \cite{bdhk07} was the first to propose this approach and implement
it in practice, but its roots go back to \cite{brz06,brz08} where it was shown
that if iterates of a nonexpansive operator converge for any initial point,
then its inexact iterates with summable errors also converge. More details on
related work appear in \cite[Section 3]{cdhst14} and in \cite[Section
1]{cz-2-2014}.

\textbf{Paper structure}. The paper is laid out as follows. Section
\ref{sect:SM} presents the SM. Preliminaries needed for our study are
presented in Section \ref{sect:prelims}, and the superiorized version of the
DSAP algorithm is given in Section \ref{sect:super-DSAP}. The proof of our
main result that gives a mathematical basis for the SM is presented in Section
\ref{sect:main-result}. Conclusions are given in Section
\ref{sect:conclusions}.

\section{The Superiorization Methodology\label{sect:SM}}

Consider some mathematically-formulated problem, of any kind or sort, and
denote it by $T.$ The set of solutions, called the solution set of $T$, is
denoted by $\operatorname*{SOL}(T).$ The superiorization methodology (SM) of
\cite{cdh10,hgdc12,cdhst14} is intended for constrained minimization (CM)
problems of the form:%
\begin{equation}
\mathrm{minimize}\left\{  \phi(x)\mid x\in\Psi_{T}\right\}  , \label{eq:cm}%
\end{equation}
where $\phi:\mathbb{R}^{J}\rightarrow\mathbb{R}$ is an objective function and
$\Psi_{T}\subseteq\mathbb{R}{^{J}}$ is the solution set $\Psi_{T}%
=\operatorname*{SOL}(T)$ of a problem $T$. In \cite{bdhk07, cdh10}
$\operatorname*{SOL}(T)$ was assumed to be nonempty and in later works
\cite{hgdc12,cdhst14} this assumption was removed. Here, however, we adhere to
$\Psi_{T}=\operatorname*{SOL}(T)\neq\emptyset$ throughout this paper.

To proceed with the SM, the problem $T$ can be just about any mathematical or
mathematically-formulated problem for which a \textquotedblleft
good\textquotedblright\ iterative algorithm for its solution exists which is
bounded perturbation resilient, as explained below. Two widely-used cases of
such underlying problems and their set $\Psi_{T}$ come to mind, although the
general approach is by no means restricted to those.

\begin{case}
\textrm{\label{case:1}The set $\Psi_{T}$ is the solution set of a convex
feasibility problem (CFP) of the form: find a vector $x^{\ast}\in\Psi
_{T}:=\cap_{i=1}^{I}C_{i},$ where the sets $C_{i}\subseteq\mathbb{R}^{J}$ are
closed and convex subsets of the Euclidean space $\mathbb{R}^{J}$; see, e.g.,
\cite{bb96, byrnebook, chinneck-book} or \cite[Chapter 5]{CZ97} for results
and references on the broad topic of CFPs or consult
\cite{bk13,ac89,BC11,bmr04,CEG12,cccdh10,ER11,galantai,GTH}. In such a case,
we deal in (\ref{eq:cm}) with a standard CM problem. This is the case analyzed
in this paper. }
\end{case}

\begin{case}
\textrm{\label{case:2}The set $\Psi_{T}$ is the solution set of another CM
problem which serves as the problem $T$, such as,%
\begin{equation}
\mathrm{minimize}\left\{  J(x)\mid x\in\Omega\right\}  ,\label{eq:Jxomega}%
\end{equation}
in which case we look at%
\begin{equation}
\Psi_{T}:=\left\{  x^{\ast}\in\Omega\mid J(x^{\ast})\leq J(x)\text{ for all
}x\in\Omega\right\}  ,
\end{equation}
assuming that $\Psi_{T}$ is nonempty. This case has been studied in
\cite{rand-conmath}, \cite{gh13} and \cite{tie}. }
\end{case}

In either case, or any other case for the set $\Psi_{T}$, the SM strives not
to solve (\ref{eq:cm}) but rather the task is to find a point in $\Psi_{T}$
which is \textit{superior}, i.e., has a lower, but not necessarily minimal,
value of the $\phi$ objective function value, to one returned by an algorithm
that solves the original problem $T$ alone.

This is done in the SM by first investigating the bounded perturbation
resilience of an available iterative algorithm designed to solve the original
problem $T$ and then proactively using such permitted perturbations in order
to steer the iterates of such an algorithm toward lower values of the $\phi$
objective function while not loosing the convergence to a point in $\Psi_{T}$.
See \cite{cdh10,hgdc12,cdhst14} for details. A review of
superiorization-related previous work appears in \cite[Section 3]{cdhst14}.

\section{Preliminaries\label{sect:prelims}}

Let $X$ be a Hilbert space equipped with an inner product $\left\langle
\cdot,\cdot\right\rangle ,$ which induces a complete norm $||\cdot||$. For
each $x\in X$ and each nonempty set $E\subseteq X$ define%
\begin{equation}
d(x,E):=\operatorname{inf}\{||x-y||\mid\;y\in E\},
\end{equation}
and for each $x\in X$ and each $r>0$ define the closed ball around $x$ with
radius $r$ by%
\begin{equation}
B(x;r):=\{y\in X\mid\;||x-y||\leq r\}.
\end{equation}

The following proposition and corollary are well-known, see, e.g.,
\cite{CEG12} or \cite{BC11}.

\begin{proposition}
\textrm{\textrm{\label{prop:prop1.1}Let $D$ be a nonempty, closed and convex
subset of $X$. Then, for each $x\in X$ there is a unique point $P_{D}(x)\in D$
(called the projection of $x$ onto $D$) satisfying%
\begin{equation}
||x-P_{D}(x)||=\operatorname{inf}\{||x-y||\mid\ y\in D\}.
\end{equation}
Moreover,%
\begin{equation}
||P_{D}(x)-P_{D}(y)||\leq||x-y||{\text{ for all }}x,y\in X,
\end{equation}
and for each $x\in X$ and each $z\in D$,%
\begin{equation}
\left\langle z-P_{D}(x),x-P_{D}(x)\right\rangle \leq0.
\end{equation}
} }
\end{proposition}

\begin{corollary}
\textrm{\label{cor:corollary1.1}Assume that $D$ is a nonempty, closed and
convex subset of $X$. Then, for each $x\in X$ and each $z\in D$,%
\begin{equation}
||z-P_{D}(x)||^{2}+||x-P_{D}(x)||^{2}\leq||z-x||^{2}.
\end{equation}
}
\end{corollary}

Suppose that $C_{1},C_{2},\dots,C_{m}$ are nonempty, closed and convex subsets
of $X$ where $m$ is a natural number. Set%
\begin{equation}
C:=\cap_{i=1}^{m}C_{i}, \label{eq:1.1}%
\end{equation}
and assume throughout that $C\neq\emptyset$. For $i=1,2,\dots,m,$ denote
$P_{i}:=P_{C_{i}}.$ By an \textit{index vector}, we a mean a vector
$t=(t_{1},t_{2},\dots,t_{q})$ such that for all $i=1,2,\dots,q$, $t_{i}%
\in\{1,2,\dots,m\}$ and whose \textit{length} is $\ell(t)=q.$ Define the
product of the individual projections onto the sets whose indices appear in
the index vector $t$ by%
\begin{equation}
P[t]:=P_{t_{q}}\cdots P_{t_{1}}%
\end{equation}
and call it a \textit{string operator}.

A finite set $\Omega$ of index vectors is called \textit{fit} iff for each
$i\in\{1,2,\dots,m\}$, there exists a vector $t\in\Omega$ such that $t_{s}=i$
for some $s\in\{1,2,\dots,q\}$. For each index vector $t$ the string operator
is nonexpansive, since the individual projections are, i.e.,%
\begin{equation}
||P[t](x)-P[t](y)||\leq||x-y||\text{ for all }x,y\in X,
\end{equation}
and also%
\begin{equation}
P[t](x)=x\text{ for all }x\in C.
\end{equation}

Denote by $\mathcal{M}$ the collection of all pairs $(\Omega,w)$, where
$\Omega$ is a finite fit set of index vectors and%
\begin{equation}
w:\Omega\rightarrow\left]  0,\infty\right[  \text{ is such that }\sum
_{t\in\Omega}w(t)=1. \label{eq:1.6}%
\end{equation}

A pair $(\Omega,w)\in\mathcal{M}$ and the function $w$ were called in
\cite{bdhk07} an \textit{amalgamator} and a \textit{fit weight function},
respectively. For any $(\Omega,w)\in\mathcal{M}$ define the convex combination
of the end-points of all strings defined by members of $\Omega$ by%
\begin{equation}
P_{\Omega,w}(x):=\sum_{t\in\Omega}w(t)P[t](x),\;x\in X. \label{eq:1.7}%
\end{equation}

It is easy to see that%
\begin{equation}
||P_{\Omega,w}(x)-P_{\Omega,w}(y)||\leq||x-y||{\text{ for all }}x,y\in X,
\label{eq:1.8}%
\end{equation}
and%
\begin{equation}
P_{\Omega,w}(x)=x{\text{ for all }}x\in C. \label{eq:1.9}%
\end{equation}

We will make use of the following condition, known in the literature as
bounded regularity, see \cite{bb96}, and assume throughout that it holds.

\begin{condition}
\textrm{\label{cond:A}For each $\epsilon>0$ and each $M>0$ there exists a
positive $\delta=\delta(\epsilon,M)$ such that for each $x\in B(0,M)$
satisfying $d(x,C_{i})\leq\delta$, for all $i=1,2,\dots,m,$ the inequality
$d(x,C)\leq\epsilon$ holds. }
\end{condition}

For the proof of the next proposition see, e.g., \cite[Proposition 5]{cz12}.

\begin{proposition}
\textrm{\textrm{If the space $X$ is finite-dimensional then Condition
\ref{cond:A} holds. } }
\end{proposition}

We choose an arbitrary fixed number $\Delta\in\left]  0,1/m\right[  $ and an
integer $\bar{q}\geq m$ and denote by $\mathcal{M}_{\ast}\equiv\mathcal{M}%
_{\ast}(\Delta,\bar{q})$ the set of all $(\Omega,w)\in\mathcal{M}$ such that
the lengths of the strings are bounded and the weights are all bounded away
from zero, namely,%
\begin{equation}
\mathcal{M}_{\ast}:=\{(\Omega,w)\in\mathcal{M\mid}\text{ }\ell(t)\leq\bar
{q}\text{ and }w(t)\geq\Delta\text{ for all }t\in\Omega\}. \label{eq:11516}%
\end{equation}

The dynamic string-averaging projection (DSAP) method with variable strings
and variable weights is the following algorithm.

\begin{algorithm}
\textrm{\label{alg:sap-v}$\left.  {}\right.  $\textbf{The DSAP method with
variable strings and variable weights} }

\textrm{\textbf{Initialization}: select an arbitrary $x^{0}\in X$, }

\textrm{\textbf{Iterative step}: given an iteration vector $x^{k}$ pick a pair
$(\Omega_{k},w_{k})\in\mathcal{M}_{\ast}$ and calculate the next iteration
vector $x^{k+1}$ by%
\begin{equation}
x^{k+1}=P_{\Omega_{k},w_{k}}(x^{k})\text{.\label{eq:algv}}%
\end{equation}
}
\end{algorithm}

The first prototypical string-averaging algorithmic scheme appeared in
\cite{ceh01} and subsequent work on various such algorithmic operators
includes \cite{pscr10,cz-2-2014,CS08,CS09,ct03,crombez,gordon,pen09,rhee03}.

If in the DSAP method one uses only a single index vector $t=(1,2,\dots,m)$
that includes all constraints indices then the fully-sequential Kaczmarz
cyclic projection method; see, e.g., \cite[p. 220]{CEG12}, is obtained,
sometimes called the POCS, for Projections Onto Convex Sets, method; see,
e.g., \cite[Chapter 5]{CZ97}. For linear hyperplanes as constraints sets the
latter is equivalent with the, independently discovered, ART, for Algebraic
Reconstruction Technique, in image reconstruction from projections, see
\cite{GTH}. If, at the other extreme, one uses exactly $m$ one-dimensional
index vectors $t=(i),$ for $i=1,2,\dots,m,$ each consisting of exactly one
constraint index, then the fully-simultaneous projection method of Cimmino;
see, e.g., \cite[page 405]{bb96}, is recovered. In-between these
\textquotedblleft extremes\textquotedblright\ the DSAP method allows for a
large \textquotedblleft arsenal\textquotedblright\ of many specific
feasibility-seeking projection algorithms -- to all of which the results of
this paper will apply.

For the reader's convenience we quote here the definition of bounded
perturbations resilience and the bounded perturbations resilience theorem of
the DSAP method, see \cite{cz12} for details. The next definition was
originally given in \cite[Definition 1]{cdh10} with a finite-dimensional
Euclidean space $\mathbb{R}^{J}$ instead of the Hilbert space $X$ in the
definition below which is taken from \cite{cz12}.

\begin{definition}
\textrm{\label{def:resilient}Given a problem $T,$ an algorithmic operator
$\mathcal{A}:X\rightarrow X$ is said to be \texttt{bounded perturbations
resilient}\emph{ }iff the following is true: if the sequence $\{x^{k}%
\}_{k=0}^{\infty},$ generated by $x^{k+1}=\mathcal{A}(x^{k}),$ for all
$k\geq0,$ converges to a solution of $T$ for all $x^{0}\in X$, then any
sequence $\{y^{k}\}_{k=0}^{\infty}$ of points in $X$ that is generated by
$y^{k+1}=\mathcal{A}(y^{k}+\beta_{k}v^{k}),$ for all $k\geq0,$ also converges
to a solution of $T$ provided that, for all $k\geq0$, $\beta_{k}v^{k}$ are
\texttt{bounded perturbations}, meaning that $\beta_{k}\geq0$ for all $k\geq0$
such that ${\displaystyle\sum\limits_{k=0}^{\infty}}\beta_{k}\,<\infty$ and
such that the sequence $\{v^{k}\}_{k=0}^{\infty}$ is bounded. }
\end{definition}

The convergence properties and the, so called, bounded perturbation resilience
of this DSAP method were analyzed in \cite{cz12}.

\begin{theorem}
\textrm{\label{thm:1.1}\cite[Theorem 12]{cz12} Let $C_{1},C_{2},\dots,C_{m}$
be nonempty, closed and convex subsets of $X,$ where $m$ is a natural number,
$C:=\cap_{i=1}^{m}C_{i}\not =\emptyset$, let $\{\beta_{k}\}_{k=0}^{\infty}$ be
a sequence of non-negative numbers such that $\sum_{k=0}^{\infty}\beta
_{k}<\infty$, let $\{v^{k}\}_{k=0}^{\infty}\subset X$ be a norm bounded
sequence, let $\{(\Omega_{k},w_{k})\}_{k=0}^{\infty}\subset\mathcal{M}_{\ast
},$ for all $k\geq0,$ and let $y^{0}\in X.$ Then, any sequence $\{y^{k}%
\}_{k=0}^{\infty},$ generated by the iterative formula%
\begin{equation}
y^{k+1}=P_{\Omega_{k},w_{k}}(y^{k}+\beta_{k}v^{k}%
)\text{,\label{eq:algv-perturbed}}%
\end{equation}
converges in the norm of $X$ and its limit belongs to $C$. }
\end{theorem}

\section{The Superiorized Version of the Dynamic String-Averaging Projection
Algorithm\label{sect:super-DSAP}}

The \textquotedblleft superiorized version of an algorithm\textquotedblright%
\ has evolved and undergone several modifications throughout the publications
on the SM, from the initial \cite[pseudocode on page 543]{bdhk07} through
\cite{dhc09,hd08,ndh12} until the most recent \cite[\textquotedblleft
Superiorized Version of the Basic Algorithm\textquotedblright\ in Section
4]{cdhst14}. The next algorithm, called \textquotedblleft The superiorized
version of the DSAP algorithm\textquotedblright, is a further modification of
the latest \cite[\textquotedblleft Superiorized Version of the Basic
Algorithm\textquotedblright\ in Section 4]{cdhst14} as we explain in Remark
\ref{remark:on-alg} below.

Let $C$ be as in (\ref{eq:1.1}), let $\phi:X\rightarrow\mathbb{R}$ be a convex
continuous function, and consider the set%
\begin{equation}
C_{min}:=\{x\in C\mid\;\phi(x)\leq\phi(y){\text{ for all }}y\in C\},
\end{equation}
and assume that $C_{min}\not =\emptyset.$

\begin{algorithm}
\textrm{\label{alg:super-dsap}$\left.  {}\right.  $\textbf{The superiorized
version of the DSAP algorithm} }

\textrm{\textbf{(0) Initialization}: Let $N$ be a natural number and let
$y^{0}\in X$ be an arbitrary user-chosen vector. }

\textrm{\textbf{(1)} \textbf{Iterative step}: Given a current vector $y^{k}$
pick an $N_{k}\in\{1,2,\dots,N\}$ and start an inner loop of calculations as
follows: }

\textrm{\textbf{(1.1) Inner loop initialization}: Define $y^{k,0}=y^{k}.$ }

\textrm{\textbf{(1.2) Inner loop step: }Given $y^{k,n},$ as long as $n<N_{k}$
do as follows: }

\textrm{\textbf{(1.2.1) }Pick a $0<\beta_{k,n}\leq1$ in a way that guarantees
that (this can be done; see Remark \ref{remark:on-alg} below) }

\textrm{%
\begin{equation}
\sum_{k=0}^{\infty}\sum_{n=0}^{N_{k}-1}\beta_{k,n}<\infty.\label{eq:2.6}%
\end{equation}
\textbf{(1.2.2) }Let $\partial\phi(y^{k,n})$ be the subgradient set of $\phi$
at $y^{k,n}$ and define $v^{k,n}$ as follows:%
\begin{equation}
v^{k,n}=\left\{
\begin{array}
[c]{cc}%
-\frac{\displaystyle s^{k,n}}{\displaystyle\left\Vert s^{k,n}\right\Vert }, &
\text{if }0\notin\partial\phi(y^{k,n}),\\
0, & \text{if }0\in\partial\phi(y^{k,n}),
\end{array}
\right.
\end{equation}
where $\displaystyle s^{k,n}\in\partial\phi(y^{k,n}).$ }

\textrm{\textbf{(1.2.3) }Calculate%
\begin{equation}
y^{k,n+1}=y^{k,n}+\beta_{k,n}v^{k,n}\label{eq:2.11}%
\end{equation}
and go to \textbf{(1.2).} }

\textrm{\textbf{(1.3) }Exit the inner loop with the vector $y^{k,N_{k}}$ }

\textrm{\textbf{(1.4) }Calculate%
\begin{equation}
y^{k+1}=\displaystyle P_{\Omega_{k},w_{k}}\displaystyle(y^{k,N_{k}%
})\label{eq:2.12}%
\end{equation}
with $(\Omega_{k},w_{k})\in\mathcal{M}_{\ast}$, and go back to \textbf{(1).} }
\end{algorithm}

\begin{remark}
\textrm{\label{remark:on-alg}$\left.  {}\right.  $ }

\begin{enumerate}
\item \textrm{For step \textbf{(1.2.1) }in\textbf{ }Algorithm
\ref{alg:super-dsap} assume that we have available a summable sequence
$\left\{  \eta_{\ell}\right\}  _{\ell=0}^{\infty}$ of positive real numbers
(for example, $\eta_{\ell}=a^{\ell}$, where $0<a<1$). Then we can let the
algorithm generate, simultaneously with the sequence $\left\{  y^{k}\right\}
_{k=0}^{\infty}$, a sequence $\left\{  \beta_{k,n}\right\}  _{k=0}^{\infty}$
as a subsequence of $\left\{  \eta_{\ell}\right\}  _{\ell=0}^{\infty}$, by
choosing $\beta_{k,n}=\eta_{\ell}$ and increasing the index $\ell$ in every
pass through step \textbf{(1.2.1) }in\textbf{ }the\textbf{ }algorithm,
resulting in a positive summable sequence $\left\{  \beta_{k,n}\right\}
_{k=0}^{\infty},$ as required in (\ref{eq:2.6}). This is how it was done in
\cite[\textquotedblleft Superiorized Version of the Basic
Algorithm\textquotedblright\ in Section 4]{cdhst14}. }

\item \textrm{There are some differences between the \textquotedblleft
Superiorized Version of the Basic Algorithm\textquotedblright\ in Section 4 of
\cite{cdhst14} and Algorithm \ref{alg:super-dsap}. In \cite{cdhst14} it was
the case that $N_{k}=N$ for all $k\geq0,$ whereas here we allow the number of
times that the inner loop step \textbf{(1.1) }is exercised to vary from
iteration to iteration depending on the iteration index $k.$ }

\item \textrm{In our Algorithm \ref{alg:super-dsap} we do not have to check if
$\phi(y^{k,n+1})\leq\phi(y^{k})$ after (\ref{eq:2.11}) in step \textbf{(1.2.3)
}of the algorithm, as is done in step 14 of the \textquotedblleft Superiorized
Version of the Basic Algorithm\textquotedblright\ in Section 4 of
\cite{cdhst14}. In spite of this saving shortcut we are able to prove the main
result for our Algorithm \ref{alg:super-dsap}, as seen in Theorem
\ref{thm:2.1} below. }

\item \textrm{Admittedly, our algorithm is related to only Case \ref{case:1}
in Section \ref{sect:SM}, and uses negative subgradients in step
\textbf{(1.2.2) }and not general nonascend steps as in step 8 of the
\textquotedblleft Superiorized Version of the Basic
Algorithm\textquotedblright\ in Section 4 of \cite{cdhst14}. }

\item \textrm{Finally, as mentioned before, our findings are related to Case
\ref{case:1} for the consistent case $C=\cap_{i=1}^{m}C_{i}\neq\emptyset$ and
treat bounded perturbations resilience and not the notion of strong
perturbation resilience as in \cite{cdhst14,hgdc12}. This enables us to prove
asymptotic convergence results here but also calls for future research to
cover the inconsistent case. }

\item \textrm{Note that the DSAP method, covered here, is a versatile
algorithmic scheme that includes, as special cases, the fully sequential
projections method and the fully simultaneous projections method as
\textquotedblleft extreme\textquotedblright\ structures obtained by putting
all sets $C_{i}$ into a single string, or by putting each constraint in a
separate string, respectively. Consult the references mentioned in Case
\ref{case:1} above for further details and relevant references. }
\end{enumerate}
\end{remark}

We will prove the following theorem as our main result.

\begin{theorem}
\textrm{\label{thm:2.1}Let $\phi:X\rightarrow\mathbb{R}$ be a convex
continuous function, and let $C_{\ast}\subseteq C_{min}$ be a nonempty subset
of $C_{min}$. Let $r_{0}\in\left]  0,1\right]  $ and $\bar{L}\geq1$ be such
that%
\begin{equation}
|\phi(x)-\phi(y)|\leq\bar{L}||x-y||{\text{ for all }}x\in C_{\ast}{\text{ and
all }}y\in B(x,r_{0}),\label{eq:2.2}%
\end{equation}
and suppose that%
\begin{equation}
\{(\Omega_{k},w_{k})\}_{k=0}^{\infty}\subset\mathcal{M}_{\ast}.\label{eq:2.3}%
\end{equation}
Then, any sequence $\{y^{k}\}_{k=0}^{\infty},$ generated by Algorithm
\ref{alg:super-dsap}, converges in the norm topology of $X$ to $y^{\ast}\in C$
and exactly one of the following two cases holds: }
\end{theorem}

(a) $y^{\ast}\in C_{min}$;

(b) $y^{\ast}\notin C_{min}$ and there exist a natural number $k_{0}$ and a
$c_{0}\in\left]  0,1\right[  $ such that for each $x\in C_{\ast}$ and each
integer $k\geq k_{0}$,%
\begin{equation}
\Vert y^{k+1}-x\Vert^{2}\leq\Vert y^{k}-x\Vert^{2}-c_{0}\sum_{n=1}^{N_{k}%
-1}\beta_{k,n}, \label{eq:in-thm-2.1}%
\end{equation}
showing that $\{y^{k}\}_{k=0}^{\infty}$ is strictly Fej\'{e}r-monotone with
respect to $C_{\ast},$ i.e., that
\begin{equation}
\Vert y^{k+1}-x\Vert^{2}<\Vert y^{k}-x\Vert^{2},\text{ for all }k\geq k_{0},
\label{eq:strict-F}%
\end{equation}
because $c_{0}\sum_{n=1}^{N_{k}-1}\beta_{k,n}>0.$

This theorem establishes a mathematical basis for the behavior of the SM when
dealing with Case \ref{case:1} in Section \ref{sect:SM}, i.e., $\Psi_{T}$ is
the solution set of a CFP as in (\ref{eq:1.1}), assuming that $\Psi_{T}%
=C\neq\emptyset,$ and using the DSAP method algorithmic scheme as a
feasibility-seeking algorithm which is indeed bounded perturbations resilient.
The bounded perturbations resilience of the DSAP method has been proved in
\cite{cz12} and the practical behavior of the SM was observed in numerous
recent works, so, we furnish here a mathematical guarantee of the convergence
behavior of the superiorized version of the DSAP Algorithm
\ref{alg:super-dsap}.

Theorem \ref{thm:2.1} tells us that any sequence $\{y^{k}\}_{k=0}^{\infty}$,
generated by Algorithm \ref{alg:super-dsap}, will not only converge to a
feasible point of the underlying CFP, which is due to the bounded
perturbations resilience of the DSAP method, but, additionally, that either
its limit point will solve the CM problem (\ref{eq:cm}) with $\Psi_{T}%
=C\neq\emptyset,$ or that the sequence $\{y^{k}\}_{k=0}^{\infty}$ is strictly
Fej\'{e}r-monotone with respect to a subset $C_{\ast}$ of the solution set
$C_{min}$ of the CM problem, according to (\ref{eq:in-thm-2.1}).

This strict Fej\'{e}r-monotonicity of the sequence $\{y^{k}\}_{k=0}^{\infty}$
does not suffice to guarantee its convergence to a minimum point of
(\ref{eq:cm}), even though the sequence does converge to a limit point in $C$,
but it says that the superiorized version algorithm retains asymptotic
convergence to a feasible point in $C,$ and that the so created
feasibility-seeking sequence has the additional property of getting strictly
closer, without necessarily converging, to a subset of the solution set of
minimizers of the CM problem (\ref{eq:cm}). For properties of
Fej\'{e}r-monotone and strictly Fej\'{e}r-monotone sequences see, e.g.,
\cite[Theorem 2.16]{bb96}, \cite[Subsection 3.3]{CEG12} and \cite{schott95}.

Another result that describes the behavior of the superiorized version of a
basic algorithm is \cite[Theorm 4.2]{cdhst14}. It considers the case of strong
perturbation resilience of the underlying basic algorithm, contrary to our
result that considers bounded perturbations resilience, and is, therefore,
valid to the case of consistent feasible sets. The inability to prove that the
superiorized version of a basic algorithm reduces the value of the objective
function $\phi,$ a fact which was repeatedly observed experimentally, made us
use the term \textquotedblleft heuristic\textquotedblright\ in
\cite{hgdc12,wenma13}. In spite of the strict Fej\'{e}r monotonicity proven
here the question of proving mathematically the reduction of the value of the
objective function $\phi$ remains.

Another open question is to formulate some reasonable further conditions that
will help distinguish before hand between the two alternatives in Theorem
\ref{thm:2.1} for the behavior of the superiorized version of a DSAP
algorithm. Published experimental results repeatedly confirm that reduction of
the value of the objective function $\phi$ is indeed achieved, without loosing
the convergence toward feasibility, see
\cite{cz12,bk13,bdhk07,cdh10,dhc09,gh13,rand-conmath,hd08,hgdc12,wenma13,ndh12,pscr10}%
. In some of these cases the SM returns a lower value of the objective
function $\phi$ than an exact minimization method with which it is compared,
e.g., \cite[Table 1]{cdhst14}.

\section{The Main Result: Strict Fej\'{e}r Monotonicity in the Superiorization
Method\label{sect:main-result}}

\subsection{Auxiliary results}

We will need the next two lemmas.

\begin{lemma}
\textrm{\label{lem:3.1}Let $x,y\in X$ and $\Delta>0$ and let $\phi
:X\rightarrow\mathbb{R}$ be a convex continuous function.%
\begin{equation}
\text{If }\phi(x)-\phi(y)>\Delta\;\text{and }v\in\partial\phi(x)\text{ then
}\left\langle v,y-x\right\rangle <-\Delta.\label{eq:3.1}%
\end{equation}
}
\end{lemma}

\begin{proof}
It follows from the subgradient inequality that,%
\begin{equation}
\left\langle v,y-x\right\rangle \leq\phi(y)-\phi(x)<-\Delta,
\end{equation}
thus proving the lemma.
\end{proof}

\begin{lemma}
\textrm{\label{lem:3.2}Under the assumptions of Theorem \ref{thm:2.1}, let
$\bar{x}\in C_{\ast}$, $\Delta\in\left]  0,r_{0}\right]  ,\;$\newline%
$\alpha\in\left]  0,1\right]  ,$ and $x\in X$ satisfy%
\begin{equation}
\phi(x)-\phi(\bar{x})>\Delta,\label{eq:3.4}%
\end{equation}
and assume that $v\in\partial\phi(x)$ and that $(\Omega,w)\in\mathcal{M}%
_{\ast}.$ Then, $v\not =0$ and%
\begin{equation}
y:=P_{\Omega,w}(x-\alpha||v||^{-1}v)\label{eq:3.6}%
\end{equation}
satisfies
\begin{equation}
\Vert y-\bar{x}\Vert^{2}\leq\Vert x-\alpha||v||^{-1}v-\bar{x}\Vert^{2}%
\leq\Vert x-\bar{x}\Vert^{2}-2\alpha(4\bar{L})^{-1}\Delta+\alpha^{2}.
\end{equation}
}
\end{lemma}

\begin{proof}
Equations (\ref{eq:3.1}) and (\ref{eq:3.4}) imply that $v\not =0$. By
(\ref{eq:2.2}), for each $z\in B(\bar{x},4^{-1}\Delta\bar{L}^{-1})$ we have%
\begin{equation}
\phi(z)-\phi(\bar{x})\leq\bar{L}||z-\bar{x}||\leq4^{-1}\Delta. \label{eq:3.7}%
\end{equation}
Thus, in view of (\ref{eq:3.7}), Lemma \ref{lem:3.1}, and (\ref{eq:3.4}), we
have for all $z\in B(\bar{x},4^{-1}\Delta\bar{L}^{-1})$,
\begin{equation}
\left\langle v,z-x\right\rangle <-(3/4)\Delta, \label{eq:3.8}%
\end{equation}
which implies that%
\begin{equation}
\left\langle ||v||^{-1}v,z-x\right\rangle <0. \label{eq:3.9}%
\end{equation}
Defining $\bar{z}:=\bar{x}+4^{-1}\bar{L}^{-1}\Delta\Vert v\Vert^{-1}v,$ we
obtain, by (\ref{eq:3.9}),%
\begin{equation}
0>\left\langle ||v||^{-1}v,\bar{z}-x\right\rangle =\left\langle ||v||^{-1}%
v,\bar{x}+4^{-1}\bar{L}^{-1}\Delta\Vert v\Vert^{-1}v-x\right\rangle ,
\end{equation}
which, in turn, yields%
\begin{equation}
\left\langle ||v||^{-1}v,\bar{x}-x\right\rangle <-4^{-1}\bar{L}^{-1}\Delta.
\label{eq:3.10}%
\end{equation}
Defining now $u:=x-\alpha\Vert v\Vert^{-1}v,$ (\ref{eq:3.10}) gives rise to%
\begin{align}
\Vert u-\bar{x}\Vert^{2}  &  =\Vert x-\alpha\Vert v\Vert^{-1}v-\bar{x}%
\Vert^{2}\nonumber\\
&  =\Vert x-\bar{x}\Vert^{2}-2\left\langle x-\bar{x},\alpha\Vert v\Vert
^{-1}v\right\rangle +\alpha^{2}\nonumber\\
&  \leq\Vert x-\bar{x}\Vert^{2}-2\alpha(4\bar{L})^{-1}\Delta+\alpha^{2}.
\end{align}
This yields, by (\ref{eq:3.6}), the definition of $u,$ (\ref{eq:1.8}),
(\ref{eq:1.9}), and the assumption that $\bar{x}\in C_{\ast}$,
\begin{align}
\Vert y-\bar{x}\Vert^{2}  &  =\Vert P_{\Omega,w}(u)-\bar{x}\Vert
^{2}\nonumber\\
&  \leq\Vert u-\bar{x}\Vert^{2}\leq\Vert x-\bar{x}\Vert^{2}-2\alpha(4\bar
{L})^{-1}\Delta+\alpha^{2},
\end{align}
which completes the proof of the lemma.
\end{proof}

\subsection{Proof of Theorem \ref{thm:2.1}}

We are now ready to prove the main result of our paper, Theorem \ref{thm:2.1}.

\begin{proof}
From Algorithm \ref{alg:super-dsap}, a sequence $\{y^{k}\}_{k=0}^{\infty}$
that it generates has the property that for each integer $k\geq1$ and each
$h\in\{1,2,\dots,N_{k}\}$,%
\begin{equation}
\Vert y^{k,h}-y^{k}\Vert=\Vert\sum_{j=1}^{h}(y^{k,j}-y^{k,j-1})\Vert\leq
\sum_{n=0}^{N_{k}-1}\Vert y^{k,n}-y^{k,n-1}\Vert\leq\sum_{n=0}^{N_{k}-1}%
\beta_{k,n} \label{eq:4.2}%
\end{equation}
and, in particular,%
\begin{equation}
\Vert y^{k,N_{k}}-y^{k}\Vert\leq\sum_{n=0}^{N_{k}-1}\beta_{k,n},
\end{equation}
so that, by (\ref{eq:2.6}),%
\begin{equation}
\sum_{k=0}^{\infty}\left\Vert y^{k,N_{k}}-y^{k}\right\Vert \leq\sum
_{k=0}^{\infty}\left(  \sum_{n=0}^{N_{k}-1}\beta_{k,n}\right)  .
\end{equation}
The bounded perturbation resilience secured by Theorem \ref{thm:1.1},
guarantees the convergence of $\{y^{k}\}_{k=0}^{\infty}$ to a point in $C,$
namely, that there exists%
\begin{equation}
y^{\ast}=\lim_{k\rightarrow\infty}y^{k}\in C \label{eq:4.5}%
\end{equation}
in the norm topology.

Assume that (a) of Theorem \ref{thm:2.1} does not hold, i.e., that $y^{\ast
}\notin C_{min}$. Then there is a $\Delta_{0}\in\left]  0,r_{0}/2\right[  $
such that%
\begin{equation}
\phi(y^{\ast})>\phi(x)+4\Delta_{0},{\text{ for all }}x\in C_{min},
\label{eq:4.7}%
\end{equation}
and there is a natural number $k_{0}$ such that for all integers $k\geq k_{0}%
$, by (\ref{eq:2.6}),%
\begin{equation}
\sum_{n=0}^{N_{k}-1}\beta_{k,n}<(16\bar{L})^{-1}\Delta_{0}. \label{eq:4.8}%
\end{equation}
This index $k_{0}$ can be chosen so that additionally for all integers $k\geq
k_{0}$ and all $\ell\in\{0,1,\dots,N_{k}\}$, by (\ref{eq:4.2}), (\ref{eq:4.5})
and (\ref{eq:4.7}),%
\begin{equation}
\phi(y^{k,\ell})>\phi(x)+2\Delta_{0},{\text{ for all }}x\in C_{min}.
\label{eq:4.9}%
\end{equation}

Next we apply Lemma \ref{lem:3.2} as follows. Pick an $\bar{x}\in C_{\ast}$,
let $k\geq k_{0}$ where $k_{0}$ is as above, and let $n\in\{1,2,\dots
,N_{k}\}.$ Use (\ref{eq:2.3}), (\ref{eq:2.11}), (\ref{eq:4.8})--(\ref{eq:4.9})
and the fact that $0<\beta_{k,n}\leq1$ in Algorithm \ref{alg:super-dsap}$,$
and apply Lemma \ref{lem:3.2} with%
\begin{equation}
\alpha=\beta_{k,n-1},\;\Delta=2\Delta_{0},\;x=y^{k,n-1},\;v=v^{k,n-1}%
,\;\text{and }(\Omega,w)=(\Omega_{k},w_{k}).
\end{equation}
This leads to%
\begin{align}
\Vert y^{k,n}-\bar{x}\Vert^{2}  &  \leq\Vert y^{k,n-1}-\bar{x}\Vert^{2}%
-2\beta_{k,n-1}(4\bar{L})^{-1}2\Delta_{0}+\beta_{k,n-1}^{2}\\
&  \leq\Vert y^{k,n-1}-\bar{x}\Vert^{2}-\beta_{k,n-1}(4\bar{L})^{-1}\Delta
_{0}, \label{eq:4.12}%
\end{align}
because $-3\beta_{k,n-1}(4\bar{L})^{-1}\Delta_{0}+\beta_{k,n-1}^{2}\leq0.$ So,
for $n=N_{k},$ in view (\ref{eq:2.12}), we have%
\begin{equation}
\Vert y^{k+1}-\bar{x}\Vert\leq\Vert y^{k,N_{k}}-\bar{x}\Vert. \label{eq:4.13}%
\end{equation}
By $y^{k,0}=y^{k}$ in Algorithm \ref{alg:super-dsap}, (\ref{eq:4.12}) and
(\ref{eq:4.13}),%
\begin{align}
\Vert y^{k}-\bar{x}\Vert^{2}-\Vert y^{k+1}-\bar{x}\Vert^{2}  &  \geq\Vert
y^{k,0}-\bar{x}\Vert^{2}-\Vert y^{k,N_{k}}-\bar{x}\Vert^{2}\nonumber\\
&  =\sum_{n=0}^{N_{k}-1}(\Vert y^{k,n-1}-\bar{x}\Vert^{2}-\Vert y^{k,n}%
-\bar{x}\Vert^{2})\nonumber\\
&  \geq(4\bar{L})^{-1}\Delta_{0}\sum_{n=0}^{N_{k}-1}\beta_{k,n},
\end{align}
and%
\begin{equation}
\Vert y^{k+1}-\bar{x}\Vert^{2}\leq\Vert y^{k}-\bar{x}\Vert^{2}-(4\bar{L}%
)^{-1}\Delta_{0}\sum_{n=0}^{N_{k}-1}\beta_{k,n},
\end{equation}
which completes the proof of Theorem \ref{thm:2.1}.
\end{proof}

\section{Conclusions\label{sect:conclusions}}

In very general terms, the superiorization methodology works by taking an
iterative algorithm, investigating its perturbation resilience, and then using
proactively such perturbations in order to \textquotedblleft
force\textquotedblright\ the perturbed algorithm to do something useful. The
perturbed algorithm is called the \textquotedblleft superiorized
version\textquotedblright\ of the original unperturbed algorithm.

If the original algorithm is efficient and useful for what it is designed to
do (computationally efficient and useful in terms of the application at hand),
and if the perturbations are simple and not expensive to calculate, then the
advantage of this method is that, for essentially the computational cost of
the original algorithm, we are able to get something more by steering its
iterates according to the perturbations.

This is a very general principle, which has been successfully used in some
important practical applications and awaits to be implemented and tested in
additional fields; see, e.g., the recent papers \cite{rand-conmath,sh14}, for
applications in intensity-modulated radiation therapy and in nondestructive
testing. An important case is when (i) the original algorithm is a
feasibility-seeking algorithm, or one that strives to find
constraint-compatible points for a family of constraints, and (ii) the
perturbations that are interlaced into the original algorithm aim at reducing
(not necessarily minimizing) a given merit (objective) function. Since its
inception in 2007, the superiorization method has developed and evolved and it
seems now worthwhile to distinguish between two directions, that nourish from
the same general principle.

One is the direction when only bounded perturbation resilience is used and the
constraints are assumed to be consistent (nonempty intersection). Then one
treats the \textquotedblleft superiorized version\textquotedblright\ of the
original unperturbed algorithm actually as a recursion formula that produces
an infinite sequence of iterates and convergence questions are meant in their
asymptotic nature. This is the framework in which we work in this paper.

The second direction does not assume consistency of the constraints but uses
instead a proximity function that \textquotedblleft measures\textquotedblright%
\ the violation of the constraints. Instead of seeking asymptotic feasibility,
it looks at $\varepsilon$-compatibility and uses the notion of strong
perturbation resilience. The same core \textquotedblleft superiorized
version\textquotedblright\ of the original unperturbed algorithm might be
investigated in each of these directions, but the second is the more practical
one whereas the first makes only asymptotic statements.

We propose the terms \textquotedblleft weak superiorization\textquotedblright%
\ and \textquotedblleft strong superiorization\textquotedblright\ as a
nomenclature for the first and second directions, respectively.\bigskip

\textbf{Acknowledgments} We thank Gabor Herman for an enlightening discussion
about terminology and we greatly appreciate the constructive comments of two
anonymous reviewers which helped us improve the paper.

\end{document}